\documentclass[11pt,a4paper]{amsart}

\usepackage[T1]{fontenc}
\usepackage[utf8]{inputenc}
\usepackage{lmodern}
\usepackage{microtype}
\usepackage{amsmath,amssymb,amsthm}
\usepackage[margin=1.15in]{geometry}
\usepackage{xcolor}
\usepackage{hyperref}

\input{glyphtounicode}
\microtypesetup{protrusion=true,expansion=false}
\theoremstyle{plain}
\newtheorem{theorem}{Theorem}[section]
\newtheorem{lemma}[theorem]{Lemma}
\newtheorem{proposition}[theorem]{Proposition}
\newtheorem{corollary}[theorem]{Corollary}

\theoremstyle{definition}

\newtheorem{question}[theorem]{Question}

\theoremstyle{remark}
\newtheorem{remark}[theorem]{Remark}

\newcommand{\VV}{V}
\newcommand{\EE}{E}
\newcommand{\Mat}{\mathcal{M}}
\newcommand{\PP}{\mathbb{P}}
\newcommand{\EX}{\mathbb{E}}
\newcommand{\Ent}{\mathbb{H}}
\newcommand{\qbar}{\overline{q}}
\newcommand{\Bin}{\operatorname{Bin}}
\newcommand{\dg}{\partial}
\newcommand{\chil}{\chi'_{\ell}}
\newcommand{\Deltaop}{\Delta}

\title[Average vacancy of random hypergraph matchings]
  {A rate for the average vacancy of uniformly random matchings in linear hypergraphs}
\author{Anish Gupta}
\address{Independent researcher}
\email{ag2269@cantab.ac.uk}
\urladdr{https://orcid.org/0009-0008-8137-7729}
\date{7 August 2026}

\hypersetup{
  colorlinks=true,
  linkcolor=blue!45!black,
  citecolor=blue!45!black,
  urlcolor=blue!45!black,
  pdftitle={A rate for the average vacancy of uniformly random matchings in linear hypergraphs},
  pdfauthor={Anish Gupta},
  pdfsubject={Average vacancy and enumeration of matchings in regular linear hypergraphs},
  pdfkeywords={hypergraph matching, linear hypergraph, average vacancy, matching enumeration},
  pdfdisplaydoctitle=true
}

\begin{document}
\begin{abstract}
Let $M$ be chosen uniformly from all matchings of a finite linear $k$-uniform hypergraph
$H$, and let $\qbar(H)$ be the average probability that a vertex is left uncovered. If
$H$ has maximum degree $D$ and normalized average degree
$\beta=k|E(H)|/(|V(H)|D)$, then, for every fixed $k\ge2$ and uniformly in the order,
\[
  \qbar(H)\ \le\ 1-\beta+\bigl(\beta k+o_D(1)\bigr)
  \frac{\log\log D}{\log D}.
\]
The underlying estimate is the order-uniform count
\[
  \log Z(H)\ge\frac{|E(H)|}{D}
    \bigl(\log D-(k+o_D(1))\log\log D\bigr),
\]
and it also counts matchings of size $(1-o_D(1))|E(H)|/D$. We prove this by sampling
edges, deleting those incident with unusually large sampled degrees, and controlling only
the total deleted mass before applying the Molloy--Reed list edge-colouring theorem.

For $d$-regular $H$, a quantitative version of the earlier Asratian--Kuzjurin
sampling-to-counting route, using the near-perfect-matching theorem of Gould and Kelly,
sharpens the coefficient: for every fixed $k\ge3$,
\[
  \sup_H\qbar(H)\ \le\
  \bigl(\max\{3,k-1\}+o_d(1)\bigr)\frac{\log\log d}{\log d}.
\]
Kahn's stronger pointwise prediction was disproved by Lee for $k\ge3$. The qualitative
averaged conclusion was already recorded by Kahn and Kim, crediting Anders Johansson,
and also follows from the Grable--Asratian--Kuzjurin enumeration; neither source states a
rate.
\end{abstract}

\maketitle

\section{Introduction}
\label{sec:intro}

Let $H=(\VV,\EE)$ be a finite hypergraph with at least one edge: $\VV$ is a finite
nonempty set and $\EE$ is a nonempty set of $k$-element subsets of $\VV$. We write $n=|\VV|$
and $\dg(v)=\{e\in\EE:v\in e\}$ for the star at $v$, and we call $H$ \emph{$k$-uniform} if
every edge has exactly $k$ vertices, \emph{linear} if $|e\cap f|\le1$ for all distinct
$e,f\in\EE$, and \emph{$d$-regular} if $|\dg(v)|=d$ for every $v\in\VV$. A \emph{matching}
of $H$ is a set of pairwise disjoint edges; $\Mat(H)$ denotes the set of all matchings of
$H$, \emph{including the empty matching}, and
\[
  Z(H)\ :=\ |\Mat(H)| ,
\]
which is the matching polynomial of $H$ evaluated at activity $1$. Throughout, $M$ denotes
a matching drawn uniformly at random from $\Mat(H)$, that is, from the activity-one
monomer--dimer law, rather than a greedy, nibble, or maximum matching. All logarithms are
natural. For $v\in\VV$ put
\[
  q_H(v)\ :=\ \PP\bigl(v\text{ is not covered by }M\bigr),
  \qquad
  \qbar(H)\ :=\ \frac1n\sum_{v\in\VV}q_H(v).
\]

How much of $H$ does such an $M$ cover? Kahn \cite{Kahn95,Kahn97} predicted in 1995 that
for $d$-regular linear $k$-uniform $H$ one has $q_H(v)=(1+o_d(1))\,d^{-1/k}$ at
\emph{every} vertex $v$, so that $M$ is almost perfect and, moreover, uniformly so. Kahn
and Kim \cite{KahnKim98} proved this for $k=2$, together with the asymptotically correct
variance of $|M|$.

For $k\ge3$ the pointwise prediction is false. For every $k\ge3$, every
$\varepsilon>0$ and every threshold $d_0$, Lee \cite{Lee} constructs some $d>d_0$ and a
$d$-regular linear $k$-uniform hypergraph containing a vertex $v_1$ with
$q_H(v_1)>1-\frac{1+\varepsilon}{d^{k-2}}$ alongside a vertex $v_2$ with
$q_H(v_2)<\frac{1+\varepsilon}{d+1}$ \cite[Theorem 1.4]{Lee}, and shows that this range is
best possible \cite[Theorem 1.5]{Lee}. A vertex can therefore be almost surely uncovered.

The failure may nevertheless be local. A single vertex with $q_H(v)$ close to $1$ moves
the average $\qbar(H)$ by only $O(1/n)$, and the exact identity in
Proposition~\ref{prop:avgid} below says that $\qbar(H)=o(1)$ is equivalent to $M$ being
almost perfect in expectation. Lee closes \cite{Lee} by observing that it may nevertheless
be true that $\sum_v q_H(v)=o_d(n)$ and writing that he believes this.
Writing
\[
  U_k(d)\ :=\ \sup\bigl\{\,\qbar(H)\ :\ H\text{ a finite $d$-regular linear $k$-uniform
  hypergraph}\,\bigr\},
\]
this is exactly the assertion $U_k(d)\to0$ as $d\to\infty$.

The qualitative assertion is not new, and reaches us by two routes. Kahn and Kim
\cite[p.~204]{KahnKim98} record, with attribution to Anders Johansson, that Pippenger's
theorem gives
\[
  \EX|M|\ \ge\ (1-o_d(1))\frac{n}{k}
\]
for a uniformly random matching of a simple $d$-regular $k$-uniform hypergraph. Their
``simple'' means linear, and their error depends only on $d$. Thus the remark is exactly
$U_k(d)\to0$. It is attributed and unproved there and supplies no rate.

There is also an enumerative route. The exponential abundance of nearly perfect
matchings was announced by Grable in the survey of Grable and Phelps
\cite{GrablePhelps}. Asratian and Kuzjurin \cite[Theorem 2 and Corollaries 1--2,
\S3]{AsratianKuzjurin} subsequently proved that, for every asymptotic sequence of
fixed-uniformity, almost-regular hypergraphs with degree $d\to\infty$ and maximum
codegree $o(d)$,
\[
  \log Z(H)=(1+o(1))\frac{n}{k}\log d.
\]
Exact regularity and linearity meet their hypotheses. Since linearity also forces
$n\ge1+(k-1)d$, a contrary sequence to $U_k(d)\to0$, combined with the elementary
star-entropy upper bound proved below, would contradict their corollary. Their proof
samples at a polynomially small density, obtains a large matching via the local lemma and
Pippenger's theorem, and transfers that event to a count by a first-moment comparison.
Making that proof quantitative with the 2025 theorem of Gould and Kelly gives the
sharpened regular bound in Theorem~\ref{thm:nibble} below. The separate contribution of
our argument is the maximum-degree statement, the size-resolved count, and an elementary
aggregate-trimming proof of an order-uniform $O_k(\log\log d/\log d)$ rate.

In the opposite direction, Proposition~\ref{prop:pointlower} below gives the elementary
pointwise bound $q_H(v)\ge1/(|\dg(v)|+1)$. Its regular average form is Lee's
Corollary 5.1 \cite[Corollary 5.1]{Lee}.

The supremum defining $U_k(d)$ allows arbitrary order. Disjoint unions create no extra
difficulty: their average vacancy is the vertex-weighted average of the component
vacancies, so the same supremum is obtained by restricting to connected hypergraphs. The
real obstruction is that $n$ may be arbitrarily large while $d$ is fixed. Bennett and
Frieze give a sharper count when $(\log n)^{10}=o(d)$, but that hypothesis does not control
the unrestricted supremum (see the discussion after Theorem~\ref{thm:engine}).

\subsection{Results}
\label{subsec:results}

\begin{theorem}[Average coverage, uniformly in the order]
\label{thm:main}
Fix $k\ge2$. Let $H$ be a finite nonempty linear $k$-uniform hypergraph on $n$ vertices,
let $D=\Deltaop(H)$, and put
\[
  \beta(H):=\frac{k|\EE|}{nD}.
\]
Then, uniformly over $H$ as $D\to\infty$,
\[
  \qbar(H)\ \le\ 1-\beta(H)
  +\bigl(\beta(H)k+o_D(1)\bigr)\frac{\log\log D}{\log D}.
\]
In particular,
\[
  U_k(d)\ \le\ \bigl(k+o_d(1)\bigr)\frac{\log\log d}{\log d}
  \ \longrightarrow\ 0.
\]
The error terms depend only on the degree parameter, on $k$, and on the constant $c_k$
of Theorem~\ref{thm:MR}; they are independent of the order and all other features of $H$.
\end{theorem}

The regular coefficient can be improved by returning to the earlier sampling-to-counting
architecture and using a quantitative near-perfect-matching theorem in place of
Pippenger's qualitative result.

\begin{theorem}[Quantitative nibble comparison]
\label{thm:nibble}
Fix $k\ge3$, and put $\rho_k=\max\{3,k-1\}$. Then
\[
  U_k(d)\ \le\ \bigl(\rho_k+o_d(1)\bigr)
  \frac{\log\log d}{\log d}.
\]
The bound is uniform in the order of the hypergraph.
\end{theorem}

For $k=3$ this agrees with the coefficient in Theorem~\ref{thm:main}; for every
$k\ge4$ it is smaller. It applies only to the regular supremum $U_k(d)$, whereas
Theorem~\ref{thm:main} gives the maximum-degree bound and Corollary~\ref{cor:manylarge}
gives a size-resolved count.

By Proposition~\ref{prop:avgid}, every finite $d$-regular linear $k$-uniform $H$ with
$k\ge3$ satisfies
\[
  \EX|M|\ \ge\ \Bigl(1-\bigl(\rho_k+o_d(1)\bigr)
  \frac{\log\log d}{\log d}\Bigr)\frac{n}{k}.
\]
The case $k=2$ is contained in, and far weaker than, the theorem of Kahn and Kim
\cite{KahnKim98}; in fact their pointwise result gives
$U_2(d)=(1+o_d(1))d^{-1/2}$. For bipartite $d$-regular graphs, Schrijver's lower bound on perfect
matchings \cite{Schrijver} also gives a stronger order-uniform count than the one used
here. At arbitrary activity, Davies, Jenssen, Perkins and Roberts \cite{DJPR} show that
$K_{d,d}$ maximizes the dimer density among $d$-regular graphs; at activity one this
identifies the graph-case minimizers of average vacancy, complementing the pointwise
asymptotics of Kahn and Kim.

Theorem~\ref{thm:main} is a statement about a global average and carries no information
about any individual vertex, so it neither revives Kahn's pointwise law nor bears on Lee's
refinement of it \cite[Conjecture 5.2]{Lee}. What it does say is that the bad vertices
produced in \cite{Lee} can never be made plentiful.

\begin{corollary}[Density of high-vacancy vertices]
\label{cor:nodensity}
Fix $k\ge2$ and $c>0$. Every finite $d$-regular linear $k$-uniform hypergraph satisfies
\[
  \frac1n\bigl|\{v:q_H(v)\ge c\}\bigr|
  \ \le\ \frac{\qbar(H)}c.
\]
For $k\ge3$, the right-hand side is at most
\[
  \biggl(\frac{\rho_k}{c}+o_d(1)\biggr)
       \frac{\log\log d}{\log d},
\]
where $\rho_k=\max\{3,k-1\}$. For $k=2$, the same conclusion holds with $2$ in
place of $\rho_k$ by Theorem~\ref{thm:main}. In particular, no fixed positive density
of vertices can have vacancy bounded away from zero as $d\to\infty$.
\end{corollary}

\begin{proof}
The first inequality is Markov's inequality applied to the nonnegative numbers $q_H(v)$.
Apply Theorem~\ref{thm:nibble} when $k\ge3$ and Theorem~\ref{thm:main} when $k=2$.
\end{proof}

Thus no iteration of the gadget in \cite[\S2]{Lee}, or any other construction, can
produce for arbitrarily large $d$ a $d$-regular linear $k$-graph with a positive density
of vertices whose vacancy probability stays bounded away from zero.

The maximum-degree form also gives a near-regular extension.

\begin{corollary}[Near-regular hypergraphs]
\label{cor:nearregular}
Fix $k\ge2$. Suppose that $H$ is finite, nonempty, linear and $k$-uniform, and that for
some $d\to\infty$ and $0\le\gamma\le1/2$ all its degrees lie in
$[(1-\gamma)d,(1+\gamma)d]$. Then, uniformly over $H$ and $\gamma$,
\[
  \qbar(H)\ \le\ \frac{2\gamma}{1+\gamma}
  +\biggl(k\frac{1-\gamma}{1+\gamma}+o_d(1)\biggr)
       \frac{\log\log d}{\log d}.
\]
In particular $\qbar(H)\to0$ whenever $\gamma=o(1)$.
\end{corollary}

\begin{proof}
With $D=\Deltaop(H)$, the average degree is at least $(1-\gamma)d$ and
$D\le(1+\gamma)d$, so $\beta(H)\ge b:=(1-\gamma)/(1+\gamma)$. The exact bound
\eqref{eq:final} is decreasing in $\beta$ for all sufficiently large $D$, because
$C_D=o(\log D)$. Substitute $\beta=b$ there, and note that
$\log D=\log d+O(1)$ uniformly for $0\le\gamma\le1/2$.
\end{proof}

The qualitative conclusion for $\gamma=o(1)$ is also contained in the almost-regular
extension of Asratian and Kuzjurin \cite[\S3]{AsratianKuzjurin}. The new content here is
the displayed rate, uniform quantitative dependence on $\gamma$, and the underlying
maximum-degree statement, which requires no lower-degree hypothesis.

Theorem~\ref{thm:main} is deduced from an entropy upper bound on $Z(H)$, which is
classical, together with the following quantitative lower bound. Asratian--Kuzjurin
\cite[Corollary 1]{AsratianKuzjurin} gives the stronger-looking
$(1+o(1))(n/k)\log d$ asymptotic for almost-regular small-codegree sequences, but without the
displayed error rate or the maximum-degree formulation needed here.

\begin{theorem}[Uniform enumeration of matchings]
\label{thm:engine}
Fix $k\ge2$. There is a function $C_D=C_{D,k}$, depending also on the constant $c_k$ of
Theorem~\ref{thm:MR}, with
\[
  C_D\ =\ \bigl(k+o_D(1)\bigr)\log\log D,
\]
such that every finite nonempty linear $k$-uniform hypergraph $H$ of maximum degree $D$
satisfies, for all sufficiently large $D$,
\[
  \log Z(H)\ \ge\ \frac{|\EE|}{D}\bigl(\log D-C_D\bigr).
\]
For a $d$-regular $H$ this becomes
\[
  \log Z(H)\ \ge\ \frac{n}{k}
  \bigl(\log d-(k+o_d(1))\log\log d\bigr).
\]
No relation between $n$ and $D$ is assumed.
\end{theorem}

Bennett and Frieze \cite[Theorem 4(b)]{BennettFrieze} prove that a $k$-uniform
$D$-regular hypergraph on $N$
vertices in which every pair of vertices lies in at most $D\varphi$ common edges has
\[
  \#\{\text{matchings}\}\ =\ \bigl((1+o(1))D e^{1-k}\bigr)^{N/k},
\]
\emph{provided} $\varphi=o(\log^{-10}N)$ and $k=o(\log \varphi^{-1})$. For a linear
hypergraph $\varphi=1/D$, so their hypothesis reads $(\log N)^{10}=o(D)$. Where it applies,
their count is strictly sharper than Theorem~\ref{thm:engine}, the error term being a
constant rather than $\log\log D$; combined with Theorem~\ref{thm:shearer} below it already
yields $\qbar\le(k-1+\log2+o(1))/\log d$ in that regime. But the hypothesis
$(\log N)^{10}=o(D)$ gives no information when $N$ grows arbitrarily fast relative to
$D$. Theorem~\ref{thm:engine} sacrifices the sharp constant to obtain an error independent
of the order.

\subsection{How the proof goes}
\label{subsec:overview}

The two theorems are joined by an entropy inequality. Applying Shearer's inequality to the
$n$ vertex stars of $H$ --- each edge coordinate is covered exactly $k$ times, once by each
of its vertices --- bounds the entropy of a uniform matching by a sum of star entropies,
each of which splits according to whether the star's centre is covered. This gives
\[
  \log Z(H)\ \le\ \frac nk\Bigl[h\bigl(\qbar(H)\bigr)
    +\bigl(1-\qbar(H)\bigr)\log D\Bigr]
\]
(Theorem~\ref{thm:shearer}), where $h$ is the binary entropy. Since $h\le\log2$, any lower
bound of the form $\log Z(H)\ge(|\EE|/D)(\log D-C)$ controls $\qbar(H)$ in terms of the
normalized average degree $\beta=k|\EE|/(nD)$. Thus the problem reduces to counting
matchings.

The random-subhypergraph comparison already underlies the Asratian--Kuzjurin count
\cite{AsratianKuzjurin}: their polynomial-density sample is handled by the local lemma
and Pippenger's theorem, followed by a first-moment transfer to a count. We use the same
broad scheme---find a large matching in a sparse sample with positive probability, then
convert that event into a count---but make its loss explicit and independent of the
order. The new ingredients are a polylogarithmic-density sample, Molloy--Reed colour
classes, and aggregate degree trimming. Write $D=\Deltaop(H)$ and retain each edge
independently with probability
$p=a/D$, where
\[
  a=\bigl\lfloor(\log D)^k(\log\log D)^{5k}\bigr\rfloor.
\]
If the sample had maximum degree at most about $a$,
then the near-optimal list edge-colouring theorem of Molloy and Reed
(Theorem~\ref{thm:MR}) would properly colour it with $(1+o(1))a$ colours; every colour
class is a matching. Passing from ``$S$ contains a matching of size at least $s$
with probability at least $\tfrac12$'' to ``$H$ has at least $\tfrac12 p^{-s}$ matchings''
is then a one-line union bound (Lemma~\ref{lem:transfer}). The matching has size close to
$|\EE|/D$, and $\log a=(k+o(1))\log\log D$, which gives Theorem~\ref{thm:engine}.

Some sampled degrees exceed $(1+\delta)a$, so every sampled edge meeting such a vertex
must be discarded. Controlling all vertex degrees simultaneously would introduce a factor
of $n$. A Chernoff bound gives
$\PP(D_v>(1+\delta)a)\le\varepsilon_D$ for each fixed vertex, but the resulting union
bound $n\varepsilon_D$ is useless when $n$ is unrestricted.

Instead we charge only the \emph{aggregate} deleted mass. If $R$ is the number of deleted
sampled edges and $\tau=(1+\delta)a$, then
$R\le\sum_vD_v\mathbf1\{D_v>\tau\}$. A vertex of degree $d_v\le D$ has
$D_v\sim\Bin(d_v,p)$. The binomial size-bias identity and stochastic domination by
$\Bin(D-1,p)$ bound the expectation of the whole sum by
$pk|\EE|\varepsilon_D$. Markov's inequality then controls $R$ with failure probability
$\sqrt{\varepsilon_D}$, independent of $n$. The lower tail of $|S|$ is a single further
event, not a union over vertices. This aggregate estimate supplies the order-uniformity.
The scale of $a$ makes the trimming and colouring losses $O(1/\log D)$ while
$\log a=(k+o(1))\log\log D$.

\subsection{Scope}
\label{subsec:scope}

The bound is not expected to be sharp. Lee's lower bound $\qbar(H)\ge1/(d+1)$ and the
scale $d^{-1/k}$ suggested by Kahn's original pointwise prediction leave a wide gap, and
we do not determine the order of $U_k(d)$. Linearity is essential to the colouring input;
individual vertices, other activities, variance, and bounded codegree lie outside the
argument. Since the constant $c_k$ in Theorem~\ref{thm:MR} is not explicit, the proof also
supplies no numerical threshold. Appendix~\ref{app:rate} illustrates the poor finite-scale
behaviour of one simple parameter choice.

\section{Notation, an identity, and the hypotheses}
\label{sec:setup}

For $x\in[0,1]$ let
\[
  h(x)\ :=\ -x\log x-(1-x)\log(1-x)
\]
be the binary entropy, with $h(0)=h(1)=0$; $h$ is concave on $[0,1]$ and $h(x)\le\log2$.
For a discrete random variable $Y$ we write $\Ent(Y)$ for its Shannon entropy in nats, and
for a random vector $X=(X_i)_{i\in I}$ and $A\subseteq I$ we write $X_A=(X_i)_{i\in A}$. We
write $\Bin(N,p)$ for the binomial distribution and $\Deltaop(G)$ for the maximum degree of
a hypergraph $G$. Finally $\chil(G)$ denotes the list chromatic index of $G$, i.e.\ the
least $\ell$ such that $G$ admits a proper edge colouring from any assignment of lists of
size $\ell$ to its edges; a proper edge colouring of a hypergraph is one in which
intersecting edges receive distinct colours, so that every colour class is a matching.

Double counting incidences gives
\begin{equation}
\label{eq:edgecount}
  \sum_{v\in\VV}|\dg(v)|\ =\ k|\EE|.
\end{equation}
In particular, a $d$-regular $k$-uniform hypergraph has $|\EE|=nd/k$.

The following identity is what makes ``average vacancy'' and ``expected matching size''
interchangeable.

\begin{proposition}
\label{prop:avgid}
For every finite $k$-uniform hypergraph $H$ on $n$ vertices and $M$ uniform on
$\Mat(H)$,
\[
  \sum_{v\in \VV} q_H(v)\ =\ n-k\,\EX|M|,
  \qquad\text{equivalently}\qquad
  \EX|M|\ =\ \frac{n}{k}\bigl(1-\qbar(H)\bigr).
\]
\end{proposition}

\begin{proof}
Every matching covers exactly $k|M|$ vertices, since its edges are pairwise disjoint and
each has $k$ vertices. Hence, for each realisation,
$\sum_{v}\mathbf{1}\{v\text{ uncovered}\}=n-k|M|$. Take expectations and use linearity.
\end{proof}

\begin{proposition}[A universal pointwise lower bound]
\label{prop:pointlower}
For every finite hypergraph $H$, every vertex $v$ and a uniform matching $M$,
\[
  q_H(v)\ \ge\ \frac1{|\dg(v)|+1}.
\]
No uniformity, regularity or linearity is required.
\end{proposition}

\begin{proof}
Write $\dg(v)=\{e_1,\ldots,e_r\}$. For each $i$, deleting $e_i$ is an injection from
the matchings containing $e_i$ into the matchings that leave $v$ uncovered. The events
$\{e_i\in M\}$ are disjoint and their union is the event that $v$ is covered. Therefore
\[
  1-q_H(v)=\sum_{i=1}^r\PP(e_i\in M)\le r q_H(v),
\]
which rearranges to the claim.
\end{proof}

Thus $\qbar(H)\le\varepsilon$ and $\EX|M|\ge(1-\varepsilon)n/k$ are the same statement.
For a $d$-regular hypergraph, Proposition~\ref{prop:pointlower} gives
$\qbar(H)\ge1/(d+1)$, the matching-size form of \cite[Corollary 5.1]{Lee}.

For a concrete small example, the Fano plane is a $3$-regular linear $3$-graph with
$Z=8$. Its seven nonempty matchings are its seven single edges, so
$\EX|M|=7/8$ and Proposition~\ref{prop:avgid} gives $\qbar=5/8$. The asymptotic theorem
is not intended to be informative at this scale.

If $H$ is a disjoint union of $H_1,\ldots,H_r$, then
$\Mat(H)=\prod_i\Mat(H_i)$ and $\qbar(H)$ is the vertex-weighted average of the
$\qbar(H_i)$. This proves the connected-component reduction used in the introduction.

\subsection{Where the hypotheses are used}
\label{subsec:hypotheses}

Finiteness is used by the entropy and counting arguments. We take $\EE$ to be a set, so
parallel copies are not part of the model; allowing them would change both linearity and
the matching count. Exact regularity is not needed. In the trimming argument a vertex of
degree $d_v\le D$ has sampled degree $\Bin(d_v,p)$; after size biasing, the relevant tail is
dominated through $\Bin(d_v-1,p)\preceq\Bin(D-1,p)$. This is what yields the
maximum-degree and near-regular forms of Theorem~\ref{thm:main}.

Linearity is used only through Theorem~\ref{thm:MR}. Hypergraphs of larger codegree would
require a different colouring input and new bookkeeping. Finally, $M$ is uniform on all
of $\Mat(H)$, with the empty matching weighted like every other matching, and the degree
parameter tends to infinity.

\section{The entropy upper bound}
\label{sec:shearer}

The upper bound on $Z(H)$ is Shearer's inequality applied to the vertex stars, an argument
with a long lineage in this role \cite{Shearer,Radhakrishnan,LinialLuria}. Bennett and
Frieze likewise use an entropy upper bound in their enumeration theorem
\cite{BennettFrieze}. We include the short application here; Shearer's inequality itself
is the classical input stated next.

\begin{lemma}[Shearer]
\label{lem:shearer}
Let $X=(X_i)_{i\in I}$ be a random vector with $I$ finite, and let $\mathcal{A}$ be a
family of subsets of $I$ such that every $i\in I$ lies in at least $t\ge 1$ members of
$\mathcal{A}$ (with multiplicity). Then
\[
  t\,\Ent(X)\ \le\ \sum_{A\in\mathcal{A}}\Ent(X_A).
\]
\end{lemma}

\begin{theorem}[Entropy upper bound on the number of matchings]
\label{thm:shearer}
Let $H$ be a finite nonempty $k$-uniform hypergraph on $n$ vertices with maximum degree
at most $D\ge1$.
Then
\[
  \log Z(H)\ \le\ \frac{n}{k}\Bigl[h\bigl(\qbar(H)\bigr)+\bigl(1-\qbar(H)\bigr)\log D\Bigr].
\]
\end{theorem}

\begin{proof}
Let $M$ be uniform on $\Mat(H)$ and let $X=(X_e)_{e\in \EE}$ with
$X_e=\mathbf{1}\{e\in M\}$. The map $M\mapsto X$ is a bijection onto its image, so $X$ is
uniform on a set of size $Z(H)$ and
\[
  \Ent(X)\ =\ \log Z(H).
\]
Take $I=\EE$ and $\mathcal{A}=\{\dg(v):v\in \VV\}$, the family of vertex stars, listed
with multiplicity. Because $H$ is $k$-uniform, each edge $e$ lies in exactly $k$ stars,
namely those of its $k$ vertices. Lemma~\ref{lem:shearer} with $t=k$ gives
\begin{equation}
\label{eq:shearer-applied}
  k\log Z(H)\ \le\ \sum_{v\in \VV}\Ent\bigl(X_{\dg(v)}\bigr).
\end{equation}

Fix $v$. Since $M$ is a matching, at most one edge of $\dg(v)$ belongs to $M$, so
$X_{\dg(v)}$ takes values in the set consisting of the zero vector together with the
standard basis vectors indexed by $\dg(v)$. Let
$Y_v=\mathbf{1}\{X_{\dg(v)}=0\}=\mathbf{1}\{v\text{ uncovered}\}$, so
$\PP(Y_v=1)=q_H(v)$. Then $Y_v$ is a function of $X_{\dg(v)}$, whence by the chain rule
\[
  \Ent\bigl(X_{\dg(v)}\bigr)
  \ =\ \Ent(Y_v)+\Ent\bigl(X_{\dg(v)}\,\big|\,Y_v\bigr)
  \ =\ h\bigl(q_H(v)\bigr)+\bigl(1-q_H(v)\bigr)\,
        \Ent\bigl(X_{\dg(v)}\,\big|\,Y_v=0\bigr),
\]
using that $\Ent(X_{\dg(v)}\mid Y_v=1)=0$ because $X_{\dg(v)}$ is then determined.
Conditionally on $Y_v=0$ the vector $X_{\dg(v)}$ takes at most $|\dg(v)|\le D$ values, so
$\Ent(X_{\dg(v)}\mid Y_v=0)\le\log D$. Therefore
\[
  \Ent\bigl(X_{\dg(v)}\bigr)\ \le\ h\bigl(q_H(v)\bigr)+\bigl(1-q_H(v)\bigr)\log D .
\]

Summing over $v$ and inserting into \eqref{eq:shearer-applied}:
\[
  k\log Z(H)\ \le\ \sum_{v}h\bigl(q_H(v)\bigr)+\Bigl(n-\sum_v q_H(v)\Bigr)\log D
  \ \le\ n\,h\bigl(\qbar(H)\bigr)+n\bigl(1-\qbar(H)\bigr)\log D ,
\]
where the last step uses concavity of $h$ (Jensen) for the first sum and the definition of
$\qbar$ for the second. Divide by $k$.
\end{proof}

Theorem~\ref{thm:shearer} uses $k$-uniformity and the maximum-degree bound but not
regularity or linearity. It is
the only place at which the target quantity $\qbar$ enters the argument, and it is what
turns a lower bound on $Z(H)$ into an upper bound on $\qbar$. The rest of the paper
produces that lower bound.

\section{A quantitative version of the earlier sampling route}
\label{sec:nibble}

We first prove Theorem~\ref{thm:nibble}. The needed special case of Gould and Kelly's
near-perfect-matching theorem \cite[Theorem 1.4]{GouldKelly} is the following. For fixed
$k\ge3$, there are constants $A$ and, for every sufficiently small fixed $\theta>0$, a
degree threshold such that an $(n,a,\varepsilon)$-regular linear $k$-uniform hypergraph
has a matching leaving at most
\begin{equation}
\label{eq:GK-leftover}
  n B^{-1+\theta}(\log a)^A
\end{equation}
vertices uncovered whenever
\[
  1\le B\le\min\{a^{1/(k-1)},1/\varepsilon\}.
\]
Here $(n,a,\varepsilon)$-regular means that every degree lies in
$[(1-\varepsilon)a,(1+\varepsilon)a]$. This follows from their theorem by putting all
codegree bounds equal to $1$; importantly, its degree threshold is independent of $n$.

\begin{proof}[Proof of Theorem~\ref{thm:nibble}]
Fix $\zeta>0$ and choose a real $t$ with
\[
  \rho_k<t<\rho_k+\zeta,
  \qquad \rho_k=\max\{3,k-1\}.
\]
Put $L=\log d$, $a=\lfloor L^t\rfloor$ and $p=a/d$. Sample each edge independently
with probability $p$. At every vertex the sampled degree $D_v$ has distribution
$\Bin(d,p)$ and mean $a$. Set
\[
  \varepsilon=\sqrt{\frac{12\log(kd)}a}.
\]
For large $d$, the two-sided Chernoff bound gives
\[
  \PP\bigl(|D_v-a|>\varepsilon a\bigr)
  \le 2\exp(-\varepsilon^2a/3)=2(kd)^{-4}.
\]
The event at $v$ depends only on the edge trials in $\dg(v)$, so it is independent of
all but at most $(k-1)d$ of the other vertex events. The lower-bound form of the
Lov\'asz local lemma \cite[Chapter 5]{AlonSpencer}, with
$x=((k-1)d+1)^{-1}$, therefore applies: indeed, for all
sufficiently large $d$, $2(kd)^{-4}\le x(1-x)^{(k-1)d}$. It shows that with probability
at least
\begin{equation}
\label{eq:LLL-probability}
  (1-x)^n\ \ge\ \exp(-2nx)
\end{equation}
all sampled degrees lie in $[(1-\varepsilon)a,(1+\varepsilon)a]$. Notice that the
lower bound may be exponentially small in $n$; its logarithm, rather than a constant
success probability, is what the counting transfer uses.

Let
\[
  b=\min\biggl\{\frac{t}{k-1},\frac{t-1}{2}\biggr\}>1
\]
and choose $\sigma>0$ with $1+\sigma<b$. Since
\[
  a^{1/(k-1)}=L^{t/(k-1)+o(1)},
  \qquad
  \varepsilon^{-1}=L^{(t-1)/2+o(1)},
\]
we may take $B=L^{1+\sigma}$ in \eqref{eq:GK-leftover}. Choose $\theta>0$ small
enough that $(1+\sigma)(1-\theta)>1$. On the event in
\eqref{eq:LLL-probability}, Gould and Kelly then give a matching covering all but
$\nu n$ vertices, where
\[
  \nu\le B^{-1+\theta}(\log a)^A=o(L^{-1}).
\]

Let $N$ be the number of matchings of $H$ with at least
$s=(1-\nu)n/k$ edges. The same union bound as in Lemma~\ref{lem:transfer}, together
with \eqref{eq:LLL-probability}, gives
\[
  \log N\ \ge\ \frac{(1-\nu)n}{k}(L-\log a)-2nx.
\]
Combining this lower bound with Theorem~\ref{thm:shearer}, and writing
$\qbar=\qbar(H)$, yields
\[
  \qbar L
  \le h(\qbar)+\nu L+(1-\nu)\log a+2kx
  \le \log2+\nu L+\log a+2kx.
\]
Hence, uniformly over $H$,
\[
  \qbar(H)\le\bigl(t+o_d(1)\bigr)\frac{\log\log d}{\log d}.
\]
Since $\zeta>0$ was arbitrary, the result follows.
\end{proof}

The dependence on $\varepsilon$ is load-bearing here. The Alon--Kim--Spencer theorem
\cite{AlonKimSpencer} is stated for exactly regular samples and cannot simply be applied
after the local-lemma step. Vu's near-regular extension \cite{Vu} corroborates the
$1/(k-1)$ exponent in its permitted error range, while the stronger estimate of Kang,
K\"uhn, Methuku and Osthus \cite{KKMO} assumes a relation between degree and order that
is unavailable in the supremum defining $U_k(d)$. The explicit
$B\le1/\varepsilon$ bottleneck in \cite{GouldKelly} is what makes the argument above
valid for arbitrary order; it is also the source of the lower constraint $t>3$.

\section{Sparse sampling with aggregate trimming}
\label{sec:trim}

Throughout Sections~\ref{sec:trim}--\ref{sec:proof} we fix $k\ge2$ and let $H$ be a
finite nonempty linear $k$-uniform hypergraph on $n$ vertices with maximum degree
$D=\Deltaop(H)$. Set
\begin{equation}
\label{eq:params}
  L:=\log D,\qquad
  a:=\bigl\lfloor L^k(\log L)^{5k}\bigr\rfloor,\qquad
  p:=\frac{a}{D},\qquad
  \delta:=\frac1L,\qquad
  \tau:=(1+\delta)a .
\end{equation}
All asymptotic statements are as $D\to\infty$ with $k$ fixed. We assume throughout that
$D$ is large enough (depending only on $k$) that
\begin{equation}
\label{eq:dlarge}
  a\ge2,\qquad p=\frac{a}{D}<1,\qquad 0<\delta<1,\qquad
  \delta a\ge 2 ,
\end{equation}
all of which hold because $a\to\infty$ and $a=o(D)$.

Let $S\subseteq \EE$ be obtained by retaining each edge of $H$ independently with
probability $p$. Put
\[
  D_v:=|S\cap\dg(v)|,\qquad
  B:=\{v\in \VV: D_v>\tau\},
\]
and let $G$ be the hypergraph on vertex set $\VV$ whose edges are the members of $S$
meeting no vertex of $B$. Let
\[
  R:=|S|-|E(G)|
\]
be the number of deleted sampled edges, and put
\begin{equation}
\label{eq:eps}
  \varepsilon_D\ :=\ \exp\bigl(-\delta^2a/24\bigr).
\end{equation}

We use the multiplicative Chernoff bounds in the following explicit form; they are
classical, see e.g.\ \cite[\S2.1]{JLR} or \cite[Appendix A]{AlonSpencer}.

\begin{lemma}[Chernoff]
\label{lem:chernoff}
Let $X\sim\Bin(N,p)$ with $\mu=\EX X=Np$. Then
\[
\begin{aligned}
  \PP\bigl(X>(1+t)\mu\bigr)&\le \exp\Bigl(-\frac{t^2\mu}{3}\Bigr)
    &&(0<t\le1),\\
  \PP\bigl(X<(1-t)\mu\bigr)&\le \exp\Bigl(-\frac{t^2\mu}{2}\Bigr)
    &&(0<t<1).
\end{aligned}
\]
\end{lemma}

The second ingredient is the exact size-bias identity that lets us evaluate, rather than
merely bound, the expected mass sitting on high-degree vertices.

\begin{lemma}[Size bias for the binomial]
\label{lem:sizebias}
Let $X\sim\Bin(N,p)$ with $N\ge1$, let $X'\sim\Bin(N-1,p)$, and let $\tau\in\mathbb{R}$.
Then
\[
  \EX\bigl[X\,\mathbf{1}\{X>\tau\}\bigr]\ =\ Np\;\PP\bigl(1+X'>\tau\bigr).
\]
\end{lemma}

\begin{proof}
Write $X=\sum_{i=1}^{N}\xi_i$ with $\xi_i$ i.i.d.\ Bernoulli$(p)$. Then
\[
  \EX\bigl[X\mathbf{1}\{X>\tau\}\bigr]
  =\sum_{i=1}^{N}\EX\bigl[\xi_i\mathbf{1}\{X>\tau\}\bigr]
  =\sum_{i=1}^{N}p\,\PP\Bigl(1+\sum_{j\ne i}\xi_j>\tau\Bigr)
  =Np\,\PP\bigl(1+X'>\tau\bigr),
\]
where the middle step conditions on $\xi_i=1$ and uses independence.
\end{proof}

The shift by $1$ is essential: the tail event below involves
$\PP(1+X'>\tau)$, not $\PP(X'>\tau)$.

\begin{lemma}[Aggregate trimming, uniformly in the order]
\label{lem:trim}
With the notation and assumptions above, let
\[
  \mathcal{G}\ :=\ \bigl\{\,|S|\ge(1-\delta)p|\EE|\,\bigr\}
    \ \cap\ \bigl\{\,R\le pk|\EE|\sqrt{\varepsilon_D}\,\bigr\}.
\]
Then
\[
  \PP\bigl(\mathcal{G}^{\,c}\bigr)\ \le\
  \exp\bigl(-\delta^2a/2\bigr)+\sqrt{\varepsilon_D},
\]
and on the event $\mathcal{G}$,
\[
  |E(G)|\ \ge\ p|\EE|\bigl(1-\delta-k\sqrt{\varepsilon_D}\bigr)
  \qquad\text{and}\qquad
  \Deltaop(G)\ \le\ (1+\delta)a .
\]
Moreover $G$ is a finite $k$-uniform linear hypergraph.
\end{lemma}

\begin{proof}
$G$ is a subhypergraph of $H$ on the same vertex set, so it is finite, $k$-uniform and
linear. If $v\in B$, then every sampled edge at $v$ meets
$B$ (namely at $v$) and so was deleted; hence $v$ has degree $0$ in $G$. If $v\notin B$
then the degree of $v$ in $G$ is at most $D_v\le\tau=(1+\delta)a$. This holds
deterministically, not merely on $\mathcal{G}$.

$|S|\sim\Bin(|\EE|,p)$ with mean $p|\EE|$. A vertex of degree $D$ is incident with $D$
distinct edges, so $|\EE|\ge D$ and hence $p|\EE|\ge a$. Lemma~\ref{lem:chernoff}
therefore gives
\[
  \PP\bigl(|S|<(1-\delta)p|\EE|\bigr)
  \ \le\ \exp\bigl(-\delta^2p|\EE|/2\bigr)
  \ \le\ \exp\bigl(-\delta^2a/2\bigr).
\]

For the degree upper tail, let $Y\sim\Bin(D-1,p)$ and
$\mu':=\EX Y=(D-1)p=a(1-1/D)$. For large $D$ we have $a/2\le\mu'\le a$. By
\eqref{eq:dlarge}, $\delta a\ge2$, so
\[
  \tau-1=(1+\delta)a-1=a+\delta a-1\ \ge\ a+\frac{\delta a}{2}
  =\Bigl(1+\frac{\delta}{2}\Bigr)a\ \ge\ \Bigl(1+\frac{\delta}{2}\Bigr)\mu' .
\]
Applying Lemma~\ref{lem:chernoff} to $Y$ with $t=\delta/2\in(0,1]$ and then using
$\mu'\ge a/2$,
\begin{equation}
\label{eq:tail}
  \PP\bigl(1+Y>\tau\bigr)=\PP\bigl(Y>\tau-1\bigr)
  \ \le\ \exp\Bigl(-\frac{(\delta/2)^2\mu'}{3}\Bigr)
  \ \le\ \exp\Bigl(-\frac{\delta^2 a}{24}\Bigr)=\varepsilon_D .
\end{equation}

Every deleted sampled edge contains at least one vertex of $B$ and is counted at least
once in $\sum_{v\in B}D_v$. Hence, deterministically,
\[
  R\ \le\ \sum_{v\in B}D_v\ =\ \sum_{v\in \VV}D_v\,\mathbf{1}\{D_v>\tau\}.
\]
Put $d_v=|\dg(v)|\le D$. When $d_v\ge1$, Lemma~\ref{lem:sizebias}, followed by the
stochastic domination $\Bin(d_v-1,p)\preceq\Bin(D-1,p)$ and \eqref{eq:tail}, gives
\[
  \EX\bigl[D_v\mathbf{1}\{D_v>\tau\}\bigr]
  =d_vp\,\PP\bigl(1+\Bin(d_v-1,p)>\tau\bigr)
  \le d_vp\varepsilon_D .
\]
The same inequality is trivial when $d_v=0$. Summing it and using
$\sum_vd_v=k|\EE|$ gives $\EX R\le pk|\EE|\varepsilon_D$. Thus Markov's inequality
gives
\[
  \PP\bigl(R>pk|\EE|\sqrt{\varepsilon_D}\bigr)\ \le\ \sqrt{\varepsilon_D}.
\]

A union bound gives the stated estimate for $\PP(\mathcal{G}^c)$. On $\mathcal{G}$,
\[
  |E(G)|=|S|-R\ \ge\ p|\EE|(1-\delta)-pk|\EE|\sqrt{\varepsilon_D}
  \ =\ p|\EE|\bigl(1-\delta-k\sqrt{\varepsilon_D}\bigr).
\]
\end{proof}

The two exceptional probabilities depend only on $D$ and $k$, which is the uniformity in
the order that will be needed below.

\section{A large matching inside the sample}
\label{sec:colour}

Our source of a large matching in the trimmed sample is the following theorem of Molloy and
Reed, a quantitative sharpening of Kahn's $\chil\le\Delta+o(\Delta)$ for linear hypergraphs
\cite{Kahn96}.

\begin{theorem}[Molloy--Reed {\cite[Theorem 1]{MolloyReed}}]
\label{thm:MR}
For every $k$ there is a constant $c_k<\infty$, depending on $k$ only, such that every
$k$-uniform linear hypergraph $G$ with maximum degree $\Delta=\Deltaop(G)\ge1$ satisfies
\[
  \chil(G)\ \le\ \Delta+c_k\,\Delta^{1-1/k}(\log\Delta)^4 .
\]
\end{theorem}

Here, as in \cite[\S1]{MolloyReed}, \emph{linear} means that no two edges intersect in more
than one vertex; enlarging the constant if necessary, we take $c_k\ge0$. The theorem has
no regularity or order hypothesis, so it applies to our irregular sample with its isolated
vertices. The constant $c_k$ is not explicit in \cite{MolloyReed}; we therefore keep it
symbolic and obtain no explicit degree threshold. Since $\chi'(G)\le\chil(G)$ and $\chi'(G)$ is an integer, $G$ has a proper edge colouring
with at most $\lfloor\Delta+c_k\Delta^{1-1/k}(\log\Delta)^4\rfloor$ colours and there is no
rounding loss; each colour class is a matching of $G$. Kahn's weaker form would suffice for
the qualitative statement $U_k(d)\to0$, but the quantitative rate below uses the displayed
error term.

We use Theorem~\ref{thm:MR} only through the following monotone form, which lets us
substitute an upper bound for $\Delta(G)$.

\begin{lemma}
\label{lem:MRmono}
For fixed $k$ and $c_k$, the function $\psi(x)=x+c_k x^{1-1/k}(\log x)^4$ is strictly
increasing on $[1,\infty)$. Consequently, if $G$ is $k$-uniform and linear with
$1\le\Deltaop(G)\le T$, then $\chil(G)\le\psi(T)$.
\end{lemma}

\begin{proof}
For $x>1$,
$\psi'(x)=1+c_k\bigl[(1-\tfrac1k)x^{-1/k}(\log x)^4+4x^{-1/k}(\log x)^3\bigr]>0$,
since $c_k\ge0$ and $\log x>0$; continuity handles the endpoint $x=1$. The consequence
is immediate from Theorem~\ref{thm:MR}.
\end{proof}

Applying this with $T=(1+\delta)a$ costs a relative $\kappa_D$, where
\begin{equation}
\label{eq:kappa}
  \kappa_D\ :=\ c_k\,(1+\delta)^{1-1/k}\,a^{-1/k}\,\bigl(\log((1+\delta)a)\bigr)^4 ,
\end{equation}
and the loss from pigeonholing a colour class of the trimmed sample is measured by
\begin{equation}
\label{eq:eta}
  \eta_D\ :=\ 1-\frac{1-\delta-k\sqrt{\varepsilon_D}}{1+\delta+\kappa_D}
  \ =\ \frac{2\delta+\kappa_D+k\sqrt{\varepsilon_D}}{1+\delta+\kappa_D}\,,
  \qquad
  s_D\ :=\ \bigl(1-\eta_D\bigr)\frac{|\EE|}{D} .
\end{equation}
The losses $\kappa_D$ and $\eta_D$ depend only on $D,k,c_k$; the target size $s_D$
also contains the natural scale $|\EE|/D$.

\begin{lemma}
\label{lem:kappa-small}
$\kappa_D\to0$, $\delta\to0$ and $\varepsilon_D\to0$ as $D\to\infty$; consequently
$\eta_D\to0$ and, for all sufficiently large $D$, $0<\eta_D<1$.
\end{lemma}

\begin{proof}
$\delta=1/\log D\to0$ and
$\varepsilon_D=\exp(-a/(24(\log D)^2))\to0$ because
$a/(\log D)^2\to\infty$. Also $(1+\delta)^{1-1/k}\le2$ and
$\log((1+\delta)a)\le\log(2a)$, so
$\kappa_D\le 2c_k a^{-1/k}(\log 2a)^4\to0$. Finally $\eta_D\to0$ from
\eqref{eq:eta}. The inequality $\eta_D<1$ is equivalent to
$\delta+k\sqrt{\varepsilon_D}<1$, which holds for large $D$, and $\eta_D>0$ since
$\delta>0$.
\end{proof}

\begin{lemma}[A large matching in $S$]
\label{lem:largematching}
For all sufficiently large $D$ (depending only on $k$ and $c_k$),
\[
  \PP\bigl(S\text{ contains a matching of }H\text{ of size at least } s_D\bigr)
  \ \ge\ \tfrac12 ,
\]
uniformly over all finite nonempty linear $k$-uniform $H$ of maximum degree $D$.
\end{lemma}

\begin{proof}
By Lemma~\ref{lem:trim} and Lemma~\ref{lem:kappa-small},
$\PP(\mathcal{G})\ge1-\exp(-\delta^2a/2)-\sqrt{\varepsilon_D}\ge\tfrac12$ for all
$D$ large enough. It therefore suffices to show that on $\mathcal{G}$ the sample $S$
contains a matching of size at least $s_D$.

Work on $\mathcal{G}$. The lower bound in Lemma~\ref{lem:trim} is positive for large
$D$, so $G$ has an edge and $\Deltaop(G)\ge1$.

$G$ is finite, $k$-uniform and linear, and $\Deltaop(G)\le(1+\delta)a$ by
Lemma~\ref{lem:trim}. By Lemma~\ref{lem:MRmono} applied with $T=(1+\delta)a$,
\[
  \chi'(G)\ \le\ \chil(G)\ \le\
  (1+\delta)a+c_k\bigl((1+\delta)a\bigr)^{1-1/k}\bigl(\log((1+\delta)a)\bigr)^4
  \ =\ a\bigl(1+\delta+\kappa_D\bigr),
\]
the last equality being \eqref{eq:kappa} after dividing by $a$. Fix a proper edge
colouring of $G$ with at most $a(1+\delta+\kappa_D)$ colours; each colour class is a
matching of $G$. By pigeonhole some class has at least
\[
  \frac{|E(G)|}{a(1+\delta+\kappa_D)}
  \ \ge\ \frac{p|\EE|\bigl(1-\delta-k\sqrt{\varepsilon_D}\bigr)}
              {a\bigl(1+\delta+\kappa_D\bigr)}
  \ =\ \frac{|\EE|}{D}\cdot
        \frac{1-\delta-k\sqrt{\varepsilon_D}}{1+\delta+\kappa_D}
  \ =\ (1-\eta_D)\frac{|\EE|}{D}\ =\ s_D
\]
edges. Since $E(G)\subseteq S$, this class is a matching of $H$ contained in $S$.
\end{proof}

The argument uses only the total edge count and maximum degree of $G$.

\section{From existence to counting, and the proofs}
\label{sec:proof}

The following elementary transfer converts Lemma~\ref{lem:largematching} into a count.

\begin{lemma}[Sampling-to-counting transfer]
\label{lem:transfer}
Let $\Omega$ be a finite set, let $p\in(0,1)$, and let $S\subseteq\Omega$ retain each
element of $\Omega$ independently with probability $p$. Let $\mathcal{F}$ be a family of
subsets of $\Omega$, let $s\in\mathbb{R}$, and suppose that
\[
  \PP\bigl(\exists\,F\in\mathcal{F}\ \text{with}\ |F|\ge s\ \text{and}\ F\subseteq S\bigr)
  \ \ge\ c .
\]
Then
\[
  \bigl|\{F\in\mathcal{F}: |F|\ge s\}\bigr|\ \ge\ c\,p^{-s}.
\]
\end{lemma}

\begin{proof}
Let $\mathcal{F}_s=\{F\in\mathcal{F}:|F|\ge s\}$ and $N=|\mathcal{F}_s|$. By the union
bound,
\[
  c\ \le\ \PP\Bigl(\bigcup_{F\in\mathcal{F}_s}\{F\subseteq S\}\Bigr)
  \ \le\ \sum_{F\in\mathcal{F}_s}\PP(F\subseteq S)
  \ =\ \sum_{F\in\mathcal{F}_s}p^{|F|}
  \ \le\ N p^{s},
\]
where the last inequality uses $|F|\ge s$ together with $0<p<1$, which makes
$x\mapsto p^{x}$ decreasing. Rearranging gives $N\ge cp^{-s}$.
\end{proof}

\begin{proof}[Proof of Theorem~\ref{thm:engine}]
Fix $k\ge2$, let $D$ be large enough for \eqref{eq:dlarge} and
Lemmas~\ref{lem:kappa-small} and~\ref{lem:largematching}, and let $H$ be a finite
nonempty linear $k$-uniform hypergraph of maximum degree $D$.

Apply Lemma~\ref{lem:transfer} with $\Omega=\EE$, $p=a/D$, $\mathcal{F}=\Mat(H)$ and
$s=s_D$. Its hypothesis holds with $c=\tfrac12$ by Lemma~\ref{lem:largematching}, so
\begin{equation}
\label{eq:Zlower}
  Z(H)\ \ge\ \frac12\Bigl(\frac{D}{a}\Bigr)^{s_D}
  \ =\ \frac12\Bigl(\frac{D}{a}\Bigr)^{(1-\eta_D)|\EE|/D}.
\end{equation}
The same estimate holds for the number of matchings of size at least $s_D$;
passing to $Z(H)$ in \eqref{eq:Zlower} discards this size information.
Taking logarithms,
\[
  \log Z(H)\ \ge\ \frac{|\EE|}{D}(1-\eta_D)\bigl(\log D-\log a\bigr)-\log 2 .
\]
Define
\begin{equation}
\label{eq:Cd}
  C_D\ :=\ \eta_D\log D+(1-\eta_D)\log a+\log 2 .
\end{equation}
Since a vertex of maximum degree $D$ is incident with $D$ distinct edges,
$|\EE|/D\ge1$. Consequently
\[
  \frac{|\EE|}{D}\bigl(\log D-C_D\bigr)
  \ =\ \frac{|\EE|}{D}(1-\eta_D)(\log D-\log a)
       -\frac{|\EE|}{D}\log2
  \ \le\ \frac{|\EE|}{D}(1-\eta_D)(\log D-\log a)-\log 2.
\]
Combining the last two displays gives
\begin{equation}
\label{eq:engine}
  \log Z(H)\ \ge\ \frac{|\EE|}{D}\bigl(\log D-C_D\bigr).
\end{equation}
The quantity $C_D$ depends only on $D,k,c_k$, not on the order or on $H$.

To estimate $C_D$, write $L=\log D$ and recall \eqref{eq:params}--\eqref{eq:eta}.
The floor in the definition of $a$ changes its logarithm by $o(1)$, and hence
\begin{align}
  \log a
    &=k\log L+5k\log\log L+o(1)
      =\bigl(k+o(1)\bigr)\log L, \label{eq:loga}\\
  \delta L&=1, \label{eq:deltaL}\\
  \kappa_D L
    &\le 2c_k\,a^{-1/k}\bigl(\log(2a)\bigr)^4L
      =O_k\bigl((\log L)^{-1}\bigr)=o(1), \label{eq:kappaL}\\
  \sqrt{\varepsilon_D}\,L
    &=\exp\bigl(-a/(48L^2)\bigr)L=o(1). \label{eq:epsL}
\end{align}
For \eqref{eq:kappaL}, we used
$a^{-1/k}=O(L^{-1}(\log L)^{-5})$ and $\log(2a)=O_k(\log L)$. For
\eqref{eq:epsL}, $a/L^2\to\infty$ even when $k=2$. From \eqref{eq:eta} and
$1+\delta+\kappa_D\ge1$ we now get
\begin{equation}
\label{eq:etaL}
  \eta_D L\ \le\ 2\delta L+\kappa_D L+k\sqrt{\varepsilon_D}L\ \le\ 2+o(1).
\end{equation}
Therefore, using \eqref{eq:Cd}, \eqref{eq:loga}, \eqref{eq:etaL} and $\eta_D\to0$,
\begin{equation}
\label{eq:Casymp}
  C_D\ =\ \eta_DL+(1-\eta_D)\log a+\log2
  \ =\ \bigl(k+o_D(1)\bigr)\log\log D .
\end{equation}
For the particular parameters in \eqref{eq:params}, the same estimates and
\eqref{eq:eta} give the more explicit expansion
\[
  C_D=k\log\log D+5k\log\log\log D+2+\log2+o_D(1).
\]
The secondary term reflects this parameter choice and is not claimed to be optimal.
In particular $C_D=O_k(\log\log D)=o(\log D)$, and \eqref{eq:engine} is the assertion of
Theorem~\ref{thm:engine}.
\end{proof}

\begin{corollary}[Many large matchings]
\label{cor:manylarge}
With $a$, $\eta_D$ and $s_D=(1-\eta_D)|\EE|/D$ as in
\eqref{eq:params} and \eqref{eq:eta}, every $H$ covered by
Theorem~\ref{thm:engine} has at least
\[
  \frac12\Bigl(\frac{D}{a}\Bigr)^{s_D}
\]
matchings of size at least $s_D$, for all sufficiently large $D$.
In particular, a $d$-regular $H$ has exponentially many matchings of size
$(1-o_d(1))n/k$.
\end{corollary}

\begin{proof}[Proof of Theorem~\ref{thm:main}]
Keep $k,D,H$ as above, put $\beta=k|\EE|/(nD)$, and abbreviate $\qbar=\qbar(H)$. Combining
Theorem~\ref{thm:shearer} with \eqref{eq:engine},
\[
  \frac{|\EE|}{D}\bigl(\log D-C_D\bigr)\ \le\ \log Z(H)
  \ \le\ \frac{n}{k}\Bigl[h(\qbar)+(1-\qbar)\log D\Bigr].
\]
Multiplying by $k/n>0$ and rearranging,
\[
  \qbar\log D\ \le\ h(\qbar)+(1-\beta)\log D+\beta C_D
  \ \le\ \log2+(1-\beta)\log D+\beta C_D,
\]
using $h\le\log2$. Hence
\begin{equation}
\label{eq:final}
  \qbar(H)\ \le\ 1-\beta+\frac{\beta C_D+\log2}{\log D}
  \ =\ 1-\beta+\bigl(\beta k+o_D(1)\bigr)\frac{\log\log D}{\log D}
\end{equation}
by \eqref{eq:Casymp}, uniformly since $0<\beta\le1$. If $H$ is $d$-regular, then $D=d$
and $\beta=1$, which gives the asserted bound on $U_k(d)$ and its convergence to zero.
\end{proof}

\begin{remark}[Where $k$ enters]
\label{rem:kenters}
Uniformity enters through the $k$-fold star cover in Theorem~\ref{thm:shearer}, the
incidence identity $\sum_v|\dg(v)|=k|\EE|$, and the exponent $1-1/k$ in
Theorem~\ref{thm:MR}. The last of these forces $a^{-1/k}\log D$ to vanish after the
colouring loss is multiplied by $\log D$, so the leading power of $a$ must be
$(\log D)^k$. The extra factor $(\log\log D)^{5k}$ absorbs the fourth power of the
logarithm in the Molloy--Reed error. Thus $\log a=(k+o(1))\log\log D$, which is the
source of the leading constant $k$ in the maximum-degree counting route. Theorem~\ref{thm:nibble}
uses a near-perfect-matching input instead of a colour class; this replaces the
$a^{1/k}$ bottleneck by $a^{1/(k-1)}$, while two-sided degree concentration creates the
separate constraint responsible for the constant $3$.
\end{remark}

\section{Concluding remarks}
\label{sec:open}

Lee's counterexamples to Kahn's pointwise prediction must be sparse. Corollary~\ref{cor:nodensity}
quantifies this: vertices with vacancy bounded away from zero occupy at most an
$O_k(\log\log d/\log d)$ fraction of a regular linear hypergraph. The mechanism specific
to our maximum-degree result is Theorem~\ref{thm:engine}: a $d$-regular
linear $k$-graph on $n$ vertices has at least $(d/\mathrm{polylog}\,d)^{n/k}$ matchings
whatever $n$ may be, and any hypergraph with that many matchings must, by Shearer, leave
few vertices uncovered on average. The quantitative sampling route in
Theorem~\ref{thm:nibble} improves the regular coefficient, but neither argument sees an
individual vertex and the gap to the natural scale remains wide.

There is a simple reason that $d^{-1/k}$ remains the natural benchmark. If $U$ is the
set of vertices left uncovered by $M$, then
\[
  \EX\bigl|\{e\in\EE:e\subseteq U\}\bigr|=\EX|M|.
\]
Indeed, pairs $(M,e)$ with $e$ disjoint from $M$ are in bijection with nonempty matchings
$M\cup\{e\}$ carrying a distinguished edge. Thus a $d$-regular $H$ has about $n/k$
uncovered induced edges in expectation when $M$ is almost perfect. If those edges occurred
at the density suggested by $|U|/n=\qbar(H)$, then
$\qbar(H)^k|\EE|\approx |\EE|/d$, suggesting $\qbar(H)\approx d^{-1/k}$. This is only a
first-moment heuristic: Lee's examples show that uncovered vertices need not behave
independently.

\begin{question}
For fixed $k\ge3$, \emph{what is the true order of $U_k(d)$?} We know
$\frac1{d+1}\le U_k(d)\le O_k(\log\log d/\log d)$, the lower bound from
Proposition~\ref{prop:pointlower} and the upper bound from Theorem~\ref{thm:nibble}, whereas the
natural benchmark is the scale $d^{-1/k}$ suggested by Kahn's pointwise prediction
\cite{Kahn95,Kahn97}, the $k=2$ theorem of \cite{KahnKim98} and
\cite[Conjecture 5.2]{Lee}. We do not even know
whether $U_k(d)=O(1/\log d)$.
\end{question}

\begin{question}
For fixed $k\ge3$, \emph{can the $\log\log d$ in Theorem~\ref{thm:engine} be removed
uniformly in $n$?} Is
$\log Z(H)\ge(n/k)(\log d-O_k(1))$ true for all finite $d$-regular linear $k$-graphs, with
no relation between $n$ and $d$? Bennett and Frieze \cite[Theorem 4(b)]{BennettFrieze}
achieve $\log Z=(n/k)(\log d-(k-1)+o(1))$, but under $(\log n)^{10}=o(d)$. Our $\log\log d$
comes from sampling at a polylogarithmic expected degree. The power
$\Delta^{1-1/k}$ in Theorem~\ref{thm:MR} forces the leading scale
$a=(\log d)^{k+o(1)}$, while its $(\log\Delta)^4$ factor dictates the additional
slowly varying factor. Molloy and Reed themselves suggest that this polylogarithmic factor
may be removable for $k\ge3$ \cite{MolloyReed}; even such an improvement would leave the
leading $(\log d)^k$ sampling scale in this route. An explicit $c_k$ would make
Theorem~\ref{thm:main} effective; removing the $\log\log d$ seems to require a different
route to a large colour class.
\end{question}

\begin{question}
\emph{Lee's Conjecture 5.2} \cite[Conjecture 5.2]{Lee}: for $k\ge3$ and
$0<\delta,\varepsilon<1$, if $H$ is an $n$-vertex $k$-graph with all degrees in
$[(1-n^{-\delta})d,(1+n^{-\delta})d]$ for some $d>n^{\varepsilon}$ and maximum codegree at
most $n^{-\delta}d$, is $q_H(v)=(1+o_n(1))d^{-1/k}$ for every $v$? This is untouched by
Theorem~\ref{thm:main}: a global average of $o(1)$ carries no pointwise information, and
the exponent is different.
\end{question}

\begin{question}
\emph{Bounded-codegree versions.} Does the conclusion persist if linearity is relaxed to
maximum codegree $o(d)$? Recent list-colouring estimates of Gould and Kelly
\cite{GouldKelly} improve the available error when the codegree tends to infinity and may
provide the appropriate colouring input, but both that input and the trimming argument
would need to be matched to the codegree regime. Near-regularity itself is already covered
by Corollary~\ref{cor:nearregular}.
\end{question}

Beyond the first moment, nothing here speaks to the distribution of $|M|$. At $k=2$ Kahn
and Kim \cite{KahnKim98} determine its variance asymptotically. For $k\ge3$ their
Conjecture~1.4 predicts $\operatorname{Var}|M|\sim n/(k^2d^{1/k})$ alongside the now-false
pointwise prediction; the variance half remains open. The method of this paper, which
discards all but one colour class of one sample, is not a plausible route to it.

\section*{Data and code availability}
The exact verification programs are included with the arXiv submission as
ancillary files and are maintained in the companion repository at
\url{https://github.com/agupta/average-vacancy-linear-hypergraphs}.

\section*{Acknowledgment of generative-AI assistance}
Anthropic Claude Code (Claude 5 family) and OpenAI Codex (GPT-5.6 family) were
used extensively for proof exploration, software development, exact
computational checks, literature discovery, and drafting and editing the
manuscript. The author selected the arguments and methods, checked the cited
sources and reported computations, and takes full responsibility for the
content. These systems are not authors or independent guarantors of
correctness.

\appendix

\section{The explicit rate}
\label{app:rate}

Two separate obstacles stand between Theorem~\ref{thm:main} and an effective statement.
The first is that the constant $c_k$ of Theorem~\ref{thm:MR} is not given a value in
\cite{MolloyReed} and we have not attempted to extract one; since $c_k$ enters $\kappa_D$
linearly, no effective threshold $d_0(\varepsilon,k)$ follows from the argument as it
stands, and everything below is stated as a function of a hypothesised value of $c_3$.

The second is that even with $c_3=1$ the degree threshold is enormous. We compare the
parameters used in the theorem,
\[
  a_\star=\bigl\lfloor L^3(\log L)^{15}\bigr\rfloor,
  \qquad \delta=L^{-1},
\]
with the transparent specialization $a=L^4$, again with $L=\log d$ at integer $L$.
The latter is admissible but has asymptotic coefficient $4$ rather than $3$. Put
\[
  \varepsilon^{(4)}=e^{-L^2/24},\qquad
  \kappa^{(4)}=c_3(1+L^{-1})^{2/3}L^{-4/3}
    \bigl(\log((1+L^{-1})L^4)\bigr)^4,
\]
and define $\eta^{(4)}$ by \eqref{eq:eta} with these quantities. Repeating the regular
case of the proof gives
\[
  \qbar(H)\le B_4(L,c_3):=\frac{C^{(4)}+\log2}{L},\qquad
  C^{(4)}=\eta^{(4)}L+(1-\eta^{(4)})\log a+\log2.
\]
Define $B_\star$ by the same formulas with $a=a_\star$ and the quantities from
\eqref{eq:eps}--\eqref{eq:eta}. For regular linear hypergraphs the last $\log2$ term
could be reduced further using $n\ge1+(k-1)d$, but that negligible refinement is not
included. The accompanying program evaluates both bounds using exact rational arithmetic
and rigorous two-sided enclosures (Appendix~\ref{app:computations}).
The grid evaluates the continuous formula in $L$; it does not assert that $e^L$ is an
integer degree at each listed point.

\begin{center}
\begin{tabular}{lll}
\hline
$\log d$ & theorem parameters $B_\star$ & transparent specialization $B_4$\\
\hline
$10^{3}$ & $0.318<B_\star<0.319$ & $0.9<B_4<1$\\
$10^{4}$ & $0.026<B_\star<0.027$ & $B_4>0.8$\\
$10^{6}$ & $0.000168<B_\star<0.000170$ &
  $B_4\in[0.085360456,\,0.085360457]$\\
\hline
\end{tabular}
\end{center}

\noindent
The theorem's parameters are substantially better than the transparent specialization,
but even their first listed value corresponds to $d=e^{1000}$, a number with roughly
$434$ decimal digits. The unknown constant also matters: for the transparent
specialization, with $c_3=100$ the bound is above $0.9$ at $L=10^6$, and at
$L=10^{10}$ it is more than fifty times its value at $c_3=1$.

For $c_3=1$, the program also certifies that $\kappa^{(4)}L$ exceeds $4\log L$ at
$L=10^6$ and $10^{20}$ but is below it at $L=10^{25}$. At $L=10^{25}$ it certifies
$4<C^{(4)}/\log L<4.3$, and at $L=10^{30}$ it certifies
$4<C^{(4)}/\log L<4.1$. These values illustrate how slowly that specialization reaches
its asymptotic coefficient. The numerical table is not used in the proof.

\section{The accompanying computations}
\label{app:computations}

Exact rational computations check the vacancy identity and entropy bound on small
examples, the deterministic trimming relations on two instances, finite ranges of the
size-bias and domination calculations, and the enclosures in Appendix~\ref{app:rate}.
They play no part in the proofs; the accompanying repository records their exact ranges,
commands, negative controls and limitations.

\end{document}